\documentclass[english]{article}
\usepackage[T1]{fontenc}
\usepackage[latin9]{inputenc}
\usepackage{refstyle}
\usepackage{enumitem}
\usepackage{amsthm}
\usepackage{amsmath}
\usepackage{amssymb}
\usepackage{graphicx}

\makeatletter

\AtBeginDocument{\providecommand\proref[1]{\ref{pro:#1}}}
\RS@ifundefined{subref}
  {\def\RSsubtxt{section~}\newref{sub}{name = \RSsubtxt}}
  {}
\RS@ifundefined{thmref}
  {\def\RSthmtxt{theorem~}\newref{thm}{name = \RSthmtxt}}
  {}
\RS@ifundefined{lemref}
  {\def\RSlemtxt{lemma~}\newref{lem}{name = \RSlemtxt}}
  {}

\theoremstyle{plain}
\newtheorem{thm}{\protect\theoremname}[section]
  \theoremstyle{definition}
  \newtheorem{defn}[thm]{\protect\definitionname}
  \theoremstyle{remark}
  \newtheorem*{acknowledgement*}{\protect\acknowledgementname}
  \theoremstyle{remark}
  \newtheorem*{rem*}{\protect\remarkname}
  \theoremstyle{plain}
  \newtheorem{lem}[thm]{\protect\lemmaname}
  \theoremstyle{remark}
  \newtheorem{rem}[thm]{\protect\remarkname}
  \theoremstyle{definition}
  \newtheorem*{condition*}{\protect\conditionname}
  \theoremstyle{remark}
  \newtheorem*{claim*}{\protect\claimname}
  \theoremstyle{plain}
  \newtheorem{prop}[thm]{\protect\propositionname}
  \theoremstyle{plain}
  \newtheorem{cor}[thm]{\protect\corollaryname}
  \theoremstyle{plain}
  \newtheorem*{thm*}{\protect\theoremname}
  \theoremstyle{remark}
  \newtheorem{claim}[thm]{\protect\claimname}
 \newlist{casenv}{enumerate}{4}
 \setlist[casenv]{leftmargin=*,align=left,widest={iiii}}
 \setlist[casenv,1]{label={{\itshape\ \casename} \arabic*.},ref=\arabic*}
 \setlist[casenv,2]{label={{\itshape\ \casename} \roman*.},ref=\roman*}
 \setlist[casenv,3]{label={{\itshape\ \casename\ \alph*.}},ref=\alph*}
 \setlist[casenv,4]{label={{\itshape\ \casename} \arabic*.},ref=\arabic*}

\date{}
\usepackage{chngcntr}

\makeatother

\usepackage{babel}
  \providecommand{\acknowledgementname}{Acknowledgement}
  \providecommand{\claimname}{Claim}
  \providecommand{\conditionname}{Condition}
  \providecommand{\corollaryname}{Corollary}
  \providecommand{\definitionname}{Definition}
  \providecommand{\lemmaname}{Lemma}
  \providecommand{\propositionname}{Proposition}
  \providecommand{\remarkname}{Remark}
  \providecommand{\theoremname}{Theorem}
 \providecommand{\casename}{Case}
\providecommand{\theoremname}{Theorem}

\begin{document}

\newcommand{\girth}{\rm{girth}}
\newcommand{\st}{\rm{st}}
\newcommand{\St}{\rm{St}}
\newcommand{\Lk}{\rm{Lk}}
\newcommand{\id}{\rm{id}}
\newcommand{\Aut}{\rm{Aut}}
\newcommand{\Isom}{\rm{Isom}}
\newcommand{\Fix}{\rm{Fix}}
\newcommand{\Stab}{\rm{Stab}}
\newcommand{\Homeo}{\rm{Homeo}}
\newcommand{\Hs}{\mathcal{H}}
\newcommand{\Hhat}{\hat{\mathcal{H}}}
\newcommand{\hhs}{\mathfrak{h}}
\newcommand{\hhat}{\hat{\mathfrak{h}}}
\newcommand{\khs}{\mathfrak{k}}
\newcommand{\khat}{\hat{\mathfrak{k}}}

\title{On Regular CAT(0) Cube Complexes\\
{\large{and the simplicity of automorphism groups of rank-one CAT(0)
cube complexes }}}

\author{Nir Lazarovich\thanks{Supported by The Adams Fellowship Program of the Israel Academy of Sciences and Humanities.} }
\maketitle
\begin{abstract}
We provide a necessary and sufficient condition on a finite flag simplicial
complex, $L$, for which there exists a unique CAT(0) cube complex
whose vertex links are all isomorphic to $L$. We then find new examples
of such CAT(0) cube complexes and prove that their automorphism groups
are virtually simple. The latter uses a result, which we prove in
the appendix, about the simplicity of certain subgroups of the automorphism
group of a rank-one CAT(0) cube complex. This result generalizes previous
results by Tits \cite{Tit70} and by Haglund and Paulin \cite{HaPa98}.
\end{abstract}

\section{Introduction}

Over the past years CAT(0) cube complexes have played a major role
in geometric group theory and have provided many examples of interesting
group actions on CAT(0) spaces. In the search for highly symmetric
CAT(0) cube complexes -- just as for their 1-dimensional analogues,
trees -- it is natural to consider the sub-class of \emph{regular}
CAT(0) cube complexes, i.e., cube complexes with the same link at
each vertex.

More precisely, recall that a CAT(0) cube complex is a 1-connected
cube complex whose vertex links are flag simplicial complexes (see
\cite{Gro87}). Let $L$ be a fixed finite flag simplicial complex.
An \emph{$L$-cube-complex} is a cube complex whose vertex links are
all isomorphic to $L$. For all $L$, the Davis complex $D(L)$ of
the right-angled Coxeter group $W_{L}$ associated to $L$ is an example
of a CAT(0) $L$-cube-complex (see Subsection \ref{sub:Right-angled-Coxeter-groups}
for more details).

A crucial difference between general CAT(0) $L$-cube-complexes and
their 1-dimensional analogues -- regular trees -- is that they are
not necessarily unique. This naturally raises the question of determining
whether or not there is a unique CAT(0) $L$-cube-complex. This question
can also be viewed as the cube-complex analogue of a similar question
for polygonal complexes that appeared in the survey paper of Farb,
Hruska and Thomas (see \cite{FHT08}). The analogue question for
polygonal complexes has been studied in various works, showing uniqueness
for certain links on the one hand, as in \cite{Bou97,GHST14,Hag02,Laz14,Swi98,Wis96}
but also finding links for which there is a continuum of non-isomorphic
complexes on the other, as in \cite{BaBr94,Hag91}.

A unique CAT(0) $L$-cube-complex is known to exist for some links
$L$, including:
\begin{itemize}
\item Any collection of isolated vertices -- i.e., the link of a regular
tree.
\item The simplex $\Delta^{d}$ for all $d\in\mathbb{N}$ -- i.e., the link
of the cube complex consisting of one $(d+1)$-dimensional cube.
\item The cycle graph $C_{n}$ for $n\ge4$ -- i.e., the link of a regular
square tiling of the Euclidean/hyperbolic plane.
\item The complete bipartite graph $K_{n,m}$ -- i.e., the link of a product
of two regular trees (see \cite{Wis96})
\item Any trivalent, 3-arc-transitive graph (see \cite{Swi98})
\end{itemize}
In fact, the question of uniqueness for CAT(0) $L$-square-complexes
(and other polygonal complexes) -- i.e, when $L$ is a graph -- was
answered in a previous paper by the author (see \cite{Laz14}).
Moreover, a fuller characterization of the graph condition given there
together with more examples of such graphs can be found in \cite{GHST14}.

In this paper we show that the following combinatorial condition on
$L$ is necessary and sufficient for uniqueness of CAT(0) $L$-cube-complex.
\begin{defn}
The simplicial complex $L$ is \emph{superstar-transitivity }if for
any two simplices $\sigma,\sigma'$ and any isomorphism $\phi:\st_{L}(\sigma)\to\st_{L}(\sigma')$,
sending $\sigma$ to $\sigma'$, there exists an automorphism $\Phi:L\to L$
such that $\Phi|_{\st_{L}(\sigma)}=\phi$.
\end{defn}
The main theorem is thus the following.
\begin{thm}[\label{Uniqueness-of-L-cube-complexes}Uniqueness of CAT(0) $L$-cube-complexes]
 Let $L$ be a finite flag simplicial complex. The associated Davis
complex $D(L)$ is the unique CAT(0) $L$-cube-complex if and only
if $L$ is superstar-transitive.
\end{thm}
Except for the above exmples, we provide in Subsection \ref{sub:The-Kneser-complex}
a new family of examples of unique CAT(0) $L$-cube-complexes of arbitrary
dimension, namely the Kneser Complexes $K_{n}^{d}$.

As in the case of regular trees, one might expect that these unique
CAT(0) $L$-cube-complexes exhibit rich automorphism group actions.
For instance, we prove that one can extend any automorphism of a hyperplane
to an automorphism of the whole complex. 

In fact we show that the following stronger property holds for any
collection of pairwise transverse hyperplanes in a unique CAT(0) $L$-cube-complex.
\begin{defn}
Let $X$ be a CAT(0) cube complex. A set of pairwise transverse hyperplanes
$\hhat_{1},\ldots,\hhat_{d}$ satisfy the \emph{hyperplane automorphism
extension property} (HAEP) if for all $\hat{f}\in\Aut\left(\hhat_{1}\cup\ldots\cup\hhat_{d}\right)$
there exists an automorphism $f\in\Aut X$ such that $f$ stabilizes
$\hhat_{1}\cup\ldots\cup\hhat_{d}$ and $f|_{\hhat_{1}\cup\ldots\cup\hhat_{d}}=\hat{f}$.
\end{defn}
We then use it to show that certain unique CAT(0) $L$-cube-complexes
have a virtually simple automorphism group. Note that the latter generalizes
the virtual simplicity of the automorphism group of a regular tree
proved in \cite{Tit70}. Also note that this is not true in general,
for instance, the automorphism groups of certain unique regular CAT(0)
cube complexes are discrete and thus virtually Coxeter groups, which
virtually map onto free groups.

The outline of the paper is as follows:

In section 2 we set the ground for the proof of the main theorem:
we define the superstar-transitivity conditions; we introduce an inductive
method on the vertices of a CAT(0) cube complex; and recall the definition
of the Davis complexes for right-angled Coxeter groups.

In section 3 we prove the main theorem.

In section 4 we prove that the \emph{hyperplane automorphism extension
property} (HAEP) holds for the unique CAT(0) $L$-cube-complexes.

In section 5 we show how the HAEP can be used to prove that certain
automorphism groups are virtually simple. We then give some examples
of unique CAT(0) $L$-cube-complexes which have a virtually simple
automorphism group. 

In the appendix we use the rank rigidity theorem (see \cite{CaSa11})
for cube complexes to generalize previous results by Tits (\cite{Tit70})
and by Haglund and Paulin (\cite{HaPa98}) about the simplicity
of the subgroup of the automorphism group of a CAT(0) cube complex
generated by all halfspace fixators. This theorem is key in the virtual
simplicity proved in section 5. A similar result is proved by Caprace
in \cite{Cap14} for the type-preserving automorphism groups of
right-angled buildings.
\begin{acknowledgement*}
The author would like to thank Pierre-Emmanuel Caprace for carefully
reading the manuscript and for providing useful and insightful comments.
His main suggestions provide a complete analysis of the other cases
in Theorem \ref{thm:v.simplicity of square complexes}, which unfortunately
were not included in the final version of this manuscript. The author
would also like to thank his advisor, Michah Sageev, for his never-ending
encouragement and helpful advice, and Fr\'ed\'eric Haglund for some helpful
conversations and exchange of ideas. The author sincerely appreciates
and acknowledges the support he received by the Adams Fellowship Program
of the Israel Academy of Sciences and Humanities.
\end{acknowledgement*}

\section{Preliminaries}

\subsection{Superstar-transitivity and basic definitions}

Let $X$ be a cube complex, and let $L$ be a finite flag simplicial
complex. Recall the following definitions:
\begin{itemize}
\item For a vertex $x\in X^{(0)}$, \emph{the} \emph{link of $x$ in} $X$,
$\Lk(x,X)$, is the simplicial complex whose vertices are the edges
incident to $x$ and whose simplices are the collections of edges
which span cubes in $X$.
\item For $L^{\prime}\subset L$, \emph{the} \emph{open star of $L^{\prime}$
in} $L$, $\st_{L}(L^{\prime})$, is the union of all open simplices
of $L$ whose closure intersects that of $L^{\prime}$.
\item Let $e$ be the directed edge in $X$ which connects $x$ to $y$,
and let $\xi$ and $\zeta$ be the vertices corresponding to $e$
in $\Lk(x,X)$ and $\Lk(y,X)$ respectively. \emph{The} \emph{transfer
map} \emph{along} $e\in X^{(1)}$ is the isomorphism $\tau_{e}:\st_{\Lk(x,X)}(\xi)\to\st_{\Lk(y,X)}(\zeta)$
which sends a simplex incident to $\xi$ in $\Lk(x,X)$ to the simplex
which represent the same cube in $\Lk(y,X)$.
\end{itemize}
Also recall that CAT(0) cube complexes naturally carry a combinatorial
structure,\emph{ }coming from the construction of hyperplanes and
half-spaces. For the definition and more details see \cite{Sag95}.
\begin{defn}
For $k\in\mathbb{N}\cup\left\{ 0\right\} $, the complex $L$ is said
to be \emph{$\st(\Delta^{k})$-transitive} if for any pair of (not
necessarily distinct) $k$-simplices $\sigma,\sigma^{\prime}$ of
$L$ and any isomorphism $\phi:\st(\sigma)\to\st(\sigma^{\prime})$
there exists an automorphism $\Phi$ of $L$ such that $\Phi|_{\st(\sigma)}=\phi$.
\end{defn}
Note that $L$ is superstar-transitive if and only if it is $\st(\Delta^{k})$-transitive
for all $k\ge0$.
\begin{rem*}
Note that, unlike graphs, a simplicial flag complex $L$ can be superstar-transitive
because there are no non-trivial isomorphisms between stars of vertices
in $L$.
\end{rem*}

\subsection{Induction on the vertices of a CAT(0) cube complex}

In this subsection we describe the properties of a certain enumeration
of the vertices of $X$. 
\begin{defn}
Let $J$ be a (typically infinite) segment in $\mathbb{Z}$ containing
$0$. An enumeration $\left\{ x_{n}\right\} _{n\in J}$, of the vertices
of a proper CAT(0) cube complex $X$ will be called \emph{admissible}
if the following hold for all $n>0$:
\begin{enumerate}
\item There exist adjacent vertices $x_{m}$ of $x_{n}$ with $m<n$.
\item The adjacent vertices $x_{m}$ of $x_{n}$ with $m<n$ are contained
in one cube.
\item For all $i<n$ such that $x_{i}$ and $x_{n}$ share a cube, $C$,
there exists $m<n$ such that $x_{m}$ is adjacent to $x_{n}$ and
$x_{m}\in C$.
\end{enumerate}
For a cube $C$ in $X$ and an admissible enumeration $\left\{ x_{n}\right\} _{n\in J}$
we denote by $C_{\le n}=\left\{ x_{i}\,|\, i\le n,\, x_{i}\in C\right\} $.\end{defn}
\begin{lem}
\label{lem: induction on Ccc}Let $X$ be a proper CAT(0) cube complex,
and let $x_{0}$ be a vertex of $X$. Let $\left\{ x_{n}\right\} _{n\ge0}$
be an enumeration of $X^{(0)}$ such that $d(x_{n},x_{0})$ is non-decreasing
(where $d$ is length metric of $X^{(1)}$). Then $\left\{ x_{n}\right\} _{n\ge0}$
is admissible.\end{lem}
\begin{proof}
Let $n>0$. Property (1) is clear from the definition. To prove (2)
we need to show that the hyperplanes separating $x_{n}$ from its
preceding adjacent vertices are pairwise transverse. Let $s,t<n$
be distinct numbers such that $x_{s}$ and $x_{t}$ are adjacent to
$x_{n}$, and let $\hhat_{s}$ and $\hhat_{t}$ be the hyperplanes
separating $x_{n}$ from $x_{s}$ and $x_{t}$ respectively.

The hyperplane $\hhat_{s}$ separates $x_{n}$ from $x_{0}$, for
otherwise $d(x_{0},x_{s})>d(x_{0},x_{n})$. Similarly $\hhat_{t}$
separates $x_{n}$ from $x_{0}$. The hyperplane $\hhat_{s}$ does
not separate $x_{t}$ from $x_{n}$, for otherwise $x_{t}$ and $x_{n}$
will be separated by two hyperplanes, and thus will not be adjacent.
Similarly, $\hhat_{t}$ does not separate $x_{s}$ from $x_{n}$.
Hence $\hhat_{s}$ and $\hhat_{t}$ are transverse, since in each
of the four possible intersections of halfspaces corresponding to
$\hhat_{s}$ and $\hhat_{t}$ there is one of the points $x_{n},x_{s},x_{t},x_{0}$.

Let $x_{m_{1}},\ldots,x_{m_{d}}$ be all the preceding adjacent vertices
of $x_{n}$, and let $\hhat_{m_{1}},\ldots,\hhat_{m_{d}}$ be the
hyperplanes which separate $x_{n}$ from $x_{m_{1}},\ldots,x_{m_{d}}$
respectively. If $i<n$ then there exists $m$ among $m_{1},\ldots,m_{d}$
such that $\hhat_{m}$ does not separate $x_{i}$ from $x_{m}$. Hence,
$\hhat_{m}$ separates $x_{i}$ from $x_{n}$. Each hyperplane $\hhat$
which separates $x_{i}$ from $x_{m}$ must also separate $x_{i}$
from $x_{n}$. Therefore, if $x_{i}$ shares a cube, $C$, with $x_{n}$
then it cannot be separated by two non-transverse hyperplanes from
$x_{n}$. It follows that $x_{i}$ cannot be separated from $x_{m}$
by two non-intersecting hyperplanes, and thus $x_{m}\in C$.
\end{proof}

\subsection{Right-angled Coxeter groups and their Davis complex\label{sub:Right-angled-Coxeter-groups}}

We recall the construction of the Davis cube complex for the right-angled
Coxeter group associated to $L$.

We first associate to $L$ the right-angled Coxeter group $W_{L}$
given by the following presentation: 
\[
W_{L}=\left\langle \xi\in L^{(0)}|\forall\xi\in L^{(0)},\xi^{2}=1\mbox{ and }\forall\xi\sim\zeta,\left[\xi,\zeta\right]=1\right\rangle .
\]

Where $\xi\sim\zeta$ if the vertices $\xi$ and $\zeta$ are adjacent
in $L$ (note that $L$ is not the Coxeter diagram for $W_{L}$).

The Davis cube complex associated to $L$, $D(L)$, is the complex
defined by adding cubes to the Cayley graph of $W_{L}$ whenever a
$1-$skeleton of a cube appears (after identifying the bigons of the
form $\xi^{2}$ for $\xi\in L^{(0)}$). 
\begin{rem}
\label{remark on Davis complexes}Observe that the complex $D(L)$
is a 1-connected $L$-cube-complex, and the identification of $\Lk(x,D(L))$
with $L$ is canonical (the identification coming from the labeling
of the edges of the Cayley graph with the generators $L^{(0)}$).
With respect to this identification, the transfer maps along edges
of $D(L)$ are the identity maps. Also observe that any automorphism
$\Phi$ of $L$ defines an automorphism $F_{\Phi}$ of $D(L)$ which
fixes the vertex which correspond to $1\in W_{L}$ and induces the
map $\Phi$ on the link of every vertex (considered via the canonical
identification with $L$).
\end{rem}

\begin{rem}
\label{remark on Davis complexes-hyperplanes}Observe that the hyperplane
in $D(L)$ transverse to an edge $e$ labeled $\xi\in L^{(0)}$ (with
respect to the canonical labeling discussed above) is isomorphic to
$D(\Lk(\xi,L))$. This isomorphism is given by the $W_{\Lk(\xi,L)}$-equivariant
embedding $W_{\Lk(\xi,L)}\hookrightarrow\left\langle \xi\right\rangle \backslash W_{L}$.
\end{rem}

\section{Uniqueness of $L$-cube-complexes}

In this section we prove Theorem \ref{Uniqueness-of-L-cube-complexes}.
\begin{defn}
Let $\left\{ x_{n}\right\} _{n\in J}$ be an admissible enumeration
of a CAT(0) cube complex $X$. Let $Y$ be another CAT(0) cube complex.
For $n\ge0$, an $n$-admissible map between $X$ and $Y$, is a pair
$(F_{n},\left\{ F_{x_{i}}\right\} _{i\le n})$ of a map $F_{n}:\left\{ x_{i}\right\} _{i\le n}\to Y^{(0)}$,
and $F_{x_{i}}:\Lk(x_{i},X)\to\Lk(F(x_{i}),Y)$ for all $i\le n$
satisfying the following:
\begin{condition*}[compatibility condition]
For every $d$-dimensional cube $C$ in $X$ for which $C_{\le n}=\left\{ x_{i}\in C\,|\, i\le n\right\} $
is non-empty there is a combinatorial embedding $G:C\hookrightarrow Y$
such that on the set $C_{\le n}$ the map $G$ coincides with $F_{n}$.
Moreover, if $x_{i}\in C_{\le n}$ and $\sigma\in\Lk(x_{i},X)$ is
the simplex which represents $C$, then $G_{x_{i}}=F_{x_{i}}|_{\sigma}$. 
\end{condition*}
In other words, this condition states that the map $F_{n}$ extends
to cubes, and the link-maps, $F_{x_{i}}$, coincide with the link-maps
of this extension.
\end{defn}
We will prove the following useful extension lemma.
\begin{lem}
\label{lem:extension of admissible maps}Let $L$ be superstar-transitive.
Let $X,Y$ be CAT(0) cube complexes, let $\left\{ x_{n}\right\} _{n\in J}$
be an admissible enumeration of the vertices of a CAT(0) cube complex
$X$, let $n\ge0$ and let $(F_{n-1},\left\{ F_{x_{i}}\right\} _{i\le n-1})$
be an $(n-1)$-admissible map between $X$ and $Y$. If the links
of the vertices $x\in X^{(0)}\setminus\left\{ x_{i}\right\} _{i\le n-1}$
and $y\in Y^{(0)}\setminus F_{n-1}(\left\{ x_{i}\right\} _{i\le n-1})$
are isomorphic to $L$, then there exists an $n$-admissible map $(F_{n},\left\{ F_{x_{i}}\right\} _{i\le n})$
between $X$ and $Y$ that extends $\left(F_{n-1},\left\{ F_{x_{i}}\right\} _{i\le n-1}\right)$
in the obvious sense.

As a corollary we see that there exists a map $F:X\to Y$ extending
$F_{n}$.\end{lem}
\begin{proof}
Let us denote by $y_{i}=F_{n-1}(x_{i})$ for all $i<n$. $F_{n}$
extends $F_{n-1}$ thus $F_{n}(x_{i})=y_{i}$ for all $i<n$.

We first define $F(x_{n})$ as follows. Let $m<n$ be such that $x_{m}$
is adjacent to $x_{n}$. Let $\xi\in\Lk(x_{m},X)$ be the vertex corresponding
to $x_{n}$ in the link of $x_{m}$. Let $y_{n}$ be the adjacent
vertex of $y_{m}$ corresponding to $F_{x_{m}}(\xi)\in\Lk(y_{m},Y)$.
Define $F(x_{n}):=y_{n}$. Note that this does not depend on the choice
of $x_{m}$, because all other choices are contained in a cube $C$
which will be mapped to a cube of $Y$ according to the compactability
condition.

Next we define $F_{x_{n}}$ as follows. Let $x_{m_{1}},\ldots,x_{m_{d}}$
be all the adjacent vertices of $x_{n}$ such that $m_{i}<n$, $i=1,\ldots,r$.
Then $x_{m_{1}},\ldots,x_{m_{d}}$ are contained in a cube $C$, this
cube correspond to a simplex $\sigma\subset\Lk(x_{n},X)$. By the
compatability condition their images $y_{m_{1}},\ldots,y_{m_{d}}$
are adjacent to $y_{n}$ and are contained in a cube $C^{\prime}$
corresponding to a simplex $\sigma^{\prime}\subset\Lk(y_{n},Y)$.
Every simplex in $\st(\sigma)$ represents a cube in $X$ which contains
$x_{n}$ and at least one of $x_{m_{j}}$. The maps $F_{y_{m_{1}}},\ldots,F_{y_{m_{d}}}$
define an isomorphism $\phi:\st(\sigma)\to\st(\sigma^{\prime})$ which
by the compactability condition is well defined. By $\st(\Delta^{d})$-transitivity
we can extend $\phi$ to an isomorphism $F_{x_{n}}:\Lk(x_{n},X)\to\Lk(y_{n},Y)$
such that $F_{x_{n}}|_{\st(\sigma)}=\phi$. By definition, for all
$1\le j\le d$ the compatability condition is satisfied for any cube
shared by $x_{m_{j}},x_{n}$.

By admissibility, every cube $C$ which contains $x_{n}$ either contains
some adjacent vertex $x_{m_{j}}$ or $C_{\le n}=\left\{ x_{n}\right\} $.
In either case the compatability condition holds.

The corollary follows by inductively applying the above argument.
\end{proof}
We are now ready to prove Theorem \ref{Uniqueness-of-L-cube-complexes}.
\begin{proof}
We begin by proving that if $L$ is superstar-transitive then there
is a unique 1-connected $L$-cube-complex.

Let $X,Y$ be two $L$-cube-complexes. Let $\left\{ x_{n}\right\} _{n\ge0}$
be an enumeration of $X^{(0)}$ as in Lemma \ref{lem: induction on Ccc}.
Set $F(x_{0})$ to be any vertex $y_{0}$ of $Y$, and set $F_{x_{0}}:\Lk(x_{0},X)\to\Lk(y_{0},Y)$
to be an isomorphism. The pair $\left(F,\left\{ F_{x_{i}}\right\} _{0\le i\le0}\right)$
is a 0-admissible map. 

By applying Lemma \ref{lem:extension of admissible maps}, we obtain
a map $\tilde{F}:X\to Y$ between the two cube complexes. This completes
the proof of the first implication. We shall now prove that if $L$
is not superstar-transitive for all $k$ then there exist more than
one CAT(0) $L$-cube-complexes.

Assume that $k$ is the minimal non-negative integer such that $L$
is not $\st(\Delta^{k})$-transitive.

If $k=0$, let $\xi,\zeta\in L^{(0)}$ and let $\phi:\st(\xi)\to\st(\zeta)$
be an isomorphism such that there is no automorphism of $L$ extending
$\phi$. Let $X$ be the following cube complex complex. Let $\hhs_{\xi}$
(resp. $\hhs_{\zeta}$) be a half-space in $D(L)$ defined by a hyperplane
associated to $\xi$ (resp. $\zeta$). By Remark \ref{remark on Davis complexes-hyperplanes},
the hyperplane $\hhat_{\xi}$ (resp. $\hhat_{\zeta}$) can be identified
with $D(\Lk(\xi,L))$ (resp. $D(\Lk(\zeta,L))$) and thus the map
$\phi$ defines an isomorphism $F_{\phi}:\hhat_{\xi}\to\hhat_{\zeta}$
by Remark \ref{remark on Davis complexes}. Form the space $X=\hhs_{\xi}\sqcup_{F_{\phi}}\hhs_{\zeta}$,
see Figure \ref{fig:Baby case example}. The space $X$ is a 1-connected
$L$-cube-complex.

\begin{figure}
\includegraphics{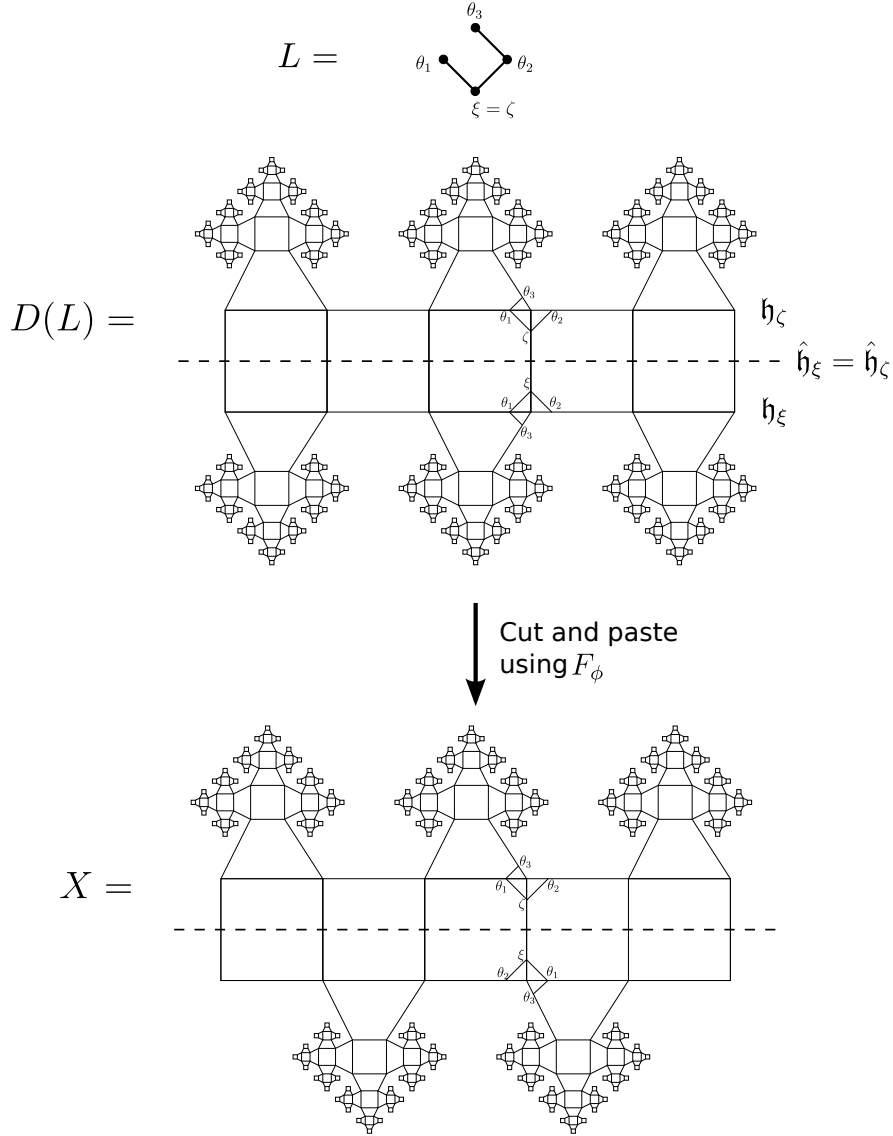}

\caption{\label{fig:Baby case example}An example of the complex $X$ for a
non-$\st(\Delta^{0})$-transitive link $L$. Let $L$ be the 3-edge
path graph. Let $\xi$ and $\zeta$ be the same vertex, shown in the
figure, and let $\phi:\st(\xi)\to\st(\zeta)$ be the non-extendible
isomorphism which exchanges $\theta_{1}$ and $\theta_{2}$. Since
$\xi=\zeta$, the hyperplanes $\hhat_{\xi}$ and $\hhat_{\zeta}$
may be chosen the same, and are isomorphic to the real line (thought
of as the Davis complex, $D(\Lk(\xi,L))$, of the infinite dihedral
group generated by the reflections $\theta_{1}$ and $\theta_{2}$).
The induced map $F_{\phi}:\hhat_{\xi}\to\hhat_{\zeta}$ is the reflection
around some vertex of $\hhat_{\xi}$. The space $X$ in the figure
is obtained by gluing the halfspaces $\hhs_{\xi}$ and $\hhs_{\zeta}$
using $F_{\phi}$.}
\end{figure}

To see that $X\ncong D(L)$ note that in $D(L)$ for each hyperplane
$\hhat$ of $D(L)$ there is a reflection fixing this hyperplane and
exchanging $\hhs$ and $\hhs^{*}$, while in $X$ the distinguished hyperplane
$\hhat_{\xi}=\hhat_{\zeta}$ does not satisfy this property since
such a reflection would imply that the induced maps on the links extend
the transfer maps, $\phi$, to an isomorphism of the links - contradicting
the assumption on $\phi$.

If $k\ge1$, let $\phi:\st(\sigma)\to\st(\sigma^{\prime})$ be an
isomorphism between the stars of the $k$-simplices $\sigma,\sigma^{\prime}\subset L$
that cannot be extended to an automorphism of $L$. Let $\sigma=\left[\xi_{0},\ldots,\xi_{k}\right]$
and let $\tilde{\sigma}=\left[\xi_{1},\ldots,\xi_{k}\right]$ be one
of its codimension-1 faces. By the minimality of $k$, the complex
$L$ is $\st(\Delta^{k-1})$-transitive, and therefore there exists
an automorphism $\tilde{\Phi}$ of $L$ extending $(\phi|_{\tilde{\sigma}})^{-1}$
(which, in particular, sends $\sigma^{\prime}$ to $\sigma$). Hence,
by post-composition with $\tilde{\Phi}$, we may assume that $\sigma=\sigma^{\prime}$,
i.e $\phi:\st(\sigma)\to\st(\sigma)$, and moreover $\phi|_{\st(\tilde{\sigma})}=\id_{\st(\tilde{\sigma})}$.
We denote $\phi_{i}=\phi|_{\st(\xi_{i})}$, which by our assumption
is the identity map for all $i>0$.

Let $\hhs_{0},\ldots,\hhs_{k}$ be a collection of halfspace in $D(L)$
which correspond to $\xi_{0},\ldots,\xi_{k}$ respectively, and which
have pairwise intersecting supporting hyperplanes. Let $\Sigma=\hhs_{0}\cap\ldots\cap\hhs_{k}$
be the sector they define. Each $\phi_{i}$ defines an automorphism
$F_{i}$ of $\hhat_{i}$. These automorphisms coincide on $\hhat_{i}\cap\hhat_{j}$
thus define an automorphism $F$ of the boundary of this sector, i.e,
\[
\partial\Sigma=\bigcup_{i=1}^{k}\left(\hhat_{i}\cap\left(\bigcap_{j\ne i}\hhs_{j}\right)\right).
\]
Let $X$ be the space obtained by gluing $\Sigma$ to $D(L)\setminus\Sigma^{\circ}$
along $\Sigma$ with respect to $F$. The space $X$ is a 1-connected
$L$-cube-complex.

Assume for contradiction that $X\cong D(L)$. Consider the $k+1$
commuting reflections, $r_{0},\ldots,r_{k}$, with respect to the
hyperplanes $\hhat_{0},\ldots,\hhat_{k}$. Let $C$ be a cube that
intersects $\bigcap_{i}\hhat_{i}$ and corresponds to $\sigma$ .
Let $C$ be identified with $\left[0,1\right]^{k+1}$ such that $v_{0}:=\left(0,\ldots,0\right)$
is the vertex in the sector $\Sigma$, and its adjacent vertices $\left(1,0,\ldots,0\right),\ldots,\left(0,\ldots,0,1\right)$
correspond to the vertices $\xi_{0},\ldots,\xi_{k}$ in the link of
$v_{0}$. 

The reflections $r_{i}$ determine automorphisms $r_{i,v}:\Lk\left(v,X\right)\to\Lk\left(r_{i}\left(v\right),X\right)$
for $i=0,\ldots,k$ and $v\in\left\{ 0,1\right\} ^{\left\{ 0,\ldots,k\right\} }$.
These reflections have the following properties:
\begin{enumerate}
\item The maps $r_{i,v}$ fix $\sigma$ (after identifying each link with
$L$ using the natural identification).
\item The restriction $r_{i,v}|_{st(\xi_{i})}=\tau_{e_{i,v}}$ where $e_{i,v}$
as the edge connecting $v$ and $r_{i}\left(v\right)$ and $\tau_{e_{i,v}}$
is the transfer map along $e_{i,v}$. Note also that by the construction
all the transfer maps $\tau_{e_{i,v}}$ are the identity maps except
for $\tau_{e_{0,v_{0}}}=\phi$ (and $\tau_{e_{0},r_{0}(v_{0})}=\phi^{-1}$)
\item For all $i,j\in\left\{ 0,\ldots,k\right\} ,i\ne j$ and $v\in\left\{ 0,1\right\} ^{\left\{ 0,\ldots,k\right\} }$
\[
r_{j,r_{i}\left(v\right)}\circ r_{i,v}=r_{i,r_{j}\left(v\right)}\circ r_{j,v}
\]

\end{enumerate}
Let us denote $\left[n\right]:=\left\{ 1,\ldots,n\right\} $ for all
$n\in\mathbb{N}$. Let $v\mapsto\tilde{v}$ denote the embedding $\left\{ 0,1\right\} ^{\left[k\right]}\hookrightarrow\left\{ 0\right\} \times\left\{ 0,1\right\} ^{\left[k\right]}\subset\left\{ 0,1\right\} ^{\left\{ 0\right\} \cup\left[k\right]}$,
and let $R_{v}:=r_{0,\tilde{v}}$ for $v\in\left\{ 0,1\right\} ^{\left[k\right]}$. 

For an injective map $\pi:\left[m\right]\to\left[k\right]$ and $v\in\left\{ 0,1\right\} ^{\left[k\right]}$
we define the automorphism $\Phi_{\pi}^{v}$ of $L$ by induction
on $m$ in the following way: 
\begin{itemize}
\item If $m=1$, $\Phi_{\pi}^{v}=R_{r_{\pi\left(m\right)}v}^{-1}\circ R_{v}$.
\item If $m>1$, $\Phi_{\pi}^{v}=\Phi_{\pi|_{\left[m-1\right]}}^{r_{\pi\left(m-1\right)}r_{\pi\left(m\right)}v}\circ\Phi_{\pi|_{\left[m-1\right]}}^{v}$
\end{itemize}
Consider the automorphism defined by $m=k,\,\pi=\id_{\left[k\right]},\, v_{0}=\left(0,\ldots,0\right)\in\left\{ 0,1\right\} ^{\left[k\right]}$
\[
\Phi:=\Phi_{\id}^{v_{0}}
\]

We complete the proof by contradicting the assumption that $\phi$
is not extendible, using the following claim.
\begin{claim*}
The restriction $\Phi|_{\st(\sigma)}$ is $\phi$.\end{claim*}
\begin{proof}
After expanding $\Phi$ using the inductive definition we get 
\[
\Phi=R_{v_{2^{k}-1}}^{-1}\circ\ldots\circ R_{v_{1}}^{-1}\circ R_{v_{0}}
\]

Where $\left\{ v_{i}\right\} _{i=0}^{2^{k}-1}=\left\{ 0,1\right\} ^{\left[k\right]}$
(in fact, the sequence $\left\{ v_{i}\right\} _{i=0}^{2^{k}-1}$ form
a Hamiltonian cycle of the 1-skeleton of the cube $\left[0,1\right]^{\left[k\right]}$).
Thus, using the second property of the maps $r_{i,v}$ and the construction
of $X$, we get
\[
\Phi|_{\st(\xi_{0})}=R_{v_{2^{k}-1}}^{-1}|_{\st(\xi_{0})}\circ\ldots\circ R_{v_{1}}^{-1}|_{\st(\xi_{0})}\circ R_{v_{0}}|_{\st(\xi_{0})}=\id_{\st\left(\xi_{0}\right)}\circ\ldots\circ\id_{\st\left(\xi_{0}\right)}\circ\phi_{0}=\phi_{0}
\]
.

We are left to prove that $\Phi|_{\st(\xi_{i})}=\phi_{i}=\id_{\st\left(\xi_{i}\right)}$
for all $i\in\left[k\right]$. We do so by proving by induction on
$m$ that for all injective maps $\pi:\left[m\right]\to\left[k\right]$,
for all $v\in\left\{ 0,1\right\} ^{\left[k\right]}$ and for all $i\in\pi\left(\left[m\right]\right)$
we have $\Phi_{\pi}^{v}|_{\st\left(\xi_{i}\right)}=\id_{\st\left(\xi_{i}\right)}$.
In particular, we get $\Phi|_{\st(\xi_{i})}=\id_{\st\left(\xi_{i}\right)}$
for all $i\in\left[k\right]$.

For the base case, $m=1$, let $i=\pi\left(1\right)$. Property 3
provides the following relation
\[
R_{r_{i}\left(v\right)}\circ r_{i,\tilde{v}}=r_{0,r_{i}\left(v\right)}\circ r_{i,\tilde{v}}=r_{i,r_{0}\left(\tilde{v}\right)}\circ r_{0,\tilde{v}}=r_{i,r_{0}\left(\tilde{v}\right)}\circ R_{v}
\]

When restricted to $\st\left(\xi_{i}\right)$ we obtain 
\[
R_{r_{i}\left(v\right)}|_{\st\left(\xi_{i}\right)}=R_{r_{i}\left(v\right)}|_{\st\left(\xi_{i}\right)}\circ r_{i,\tilde{v}}|_{\st\left(\xi_{i}\right)}=r_{i,r_{0}\left(\tilde{v}\right)}|_{\st\left(\xi_{i}\right)}\circ R_{v}|_{\st\left(\xi_{i}\right)}=R_{v}|_{\st\left(\xi_{i}\right)}
\]

since $r_{1,\iota_{0}^{0}(\emptyset)}|_{\st\left(\xi_{1}\right)}=r_{1,\iota_{0}^{1}(\emptyset)}|_{\st\left(\xi_{1}\right)}=\id_{\st\left(\xi_{1}\right)}$
by property 2. Thus, $\Phi_{\pi}^{v}|_{\st\left(\xi_{i}\right)}=R_{r_{i}\left(v\right)}^{-1}|_{\st\left(\xi_{i}\right)}\circ R_{v}|_{\st\left(\xi_{i}\right)}=\id_{\st\left(\xi_{i}\right)}$.

Now assume $m>1$. We divide the proof of the inductive step into
3 cases:

Case 1. If $i\in\pi\left(\left[m-1\right]\right)$, then by the induction
hypothesis 
\[
\Phi_{\pi}^{v}|_{\st\left(\xi_{i}\right)}=\Phi_{\pi|_{\left[m-1\right]}}^{r_{\pi\left(m-1\right)}r_{\pi\left(m\right)}v}|_{\st\left(\xi_{i}\right)}\circ\Phi_{\pi|_{\left[m-1\right]}}^{v}|_{\st\left(\xi_{i}\right)}=\id_{\st(\xi_{i})}
\]

Case 2. If $i=\pi\left(m\right)$ and $m>2$,

\begin{eqnarray*}
\Phi_{\pi}^{v} & = & \Phi_{\pi|_{\left[m-1\right]}}^{r_{\pi\left(m-1\right)}r_{\pi\left(m\right)}v}\circ\Phi_{\pi|_{\left[m-1\right]}}^{v}\\
 & = & \left(\Phi_{\pi|_{\left[m-2\right]}}^{r_{\pi\left(m-2\right)}r_{\pi\left(m\right)}v}\circ\Phi_{\pi|_{\left[m-2\right]}}^{r_{\pi\left(m-1\right)}r_{\pi\left(m\right)}v}\right)\circ\left(\Phi_{\pi|_{\left[m-2\right]}}^{r_{\pi\left(m-2\right)}r_{\pi\left(m-1\right)}v}\circ\Phi_{\pi|_{\left[m-2\right]}}^{v}\right)
\end{eqnarray*}
If we denote $J:=\Phi_{\pi|_{\left[m-2\right]}}^{r_{\pi\left(m-2\right)}r_{\pi\left(m\right)}v}$,
then
\begin{eqnarray*}
\Phi_{\pi}^{v} & = & J\left(\left(\Phi_{\pi|_{\left[m-2\right]}}^{r_{\pi\left(m-1\right)}r_{\pi\left(m\right)}v}\circ\Phi_{\pi|_{\left[m-2\right]}}^{r_{\pi\left(m-2\right)}r_{\pi\left(m-1\right)}v}\right)\circ\left(\Phi_{\pi|_{\left[m-2\right]}}^{v}\circ\Phi_{\pi|_{\left[m-2\right]}}^{r_{\pi\left(m-2\right)}r_{\pi\left(m\right)}v}\right)\right)J^{-1}
\end{eqnarray*}

Let $\pi^{\prime}:\left[m\right]\to\left[k\right]$ be the injective
map defined by 
\[
\pi^{\prime}\left(j\right)=\begin{cases}
\pi(j) & j\le m-2\\
\pi\left(m\right) & j=m-1\\
\pi\left(m-1\right) & j=m
\end{cases}
\]
Then,
\begin{eqnarray*}
\Phi_{\pi}^{v} & = & J\left(\Phi_{\pi^{\prime}|_{\left[m-1\right]}}^{r_{\pi\left(m-2\right)}r_{\pi\left(m-1\right)}v}\Phi_{\pi^{\prime}|_{\left[m-1\right]}}^{r_{\pi\left(m-2\right)}r_{\pi\left(m\right)}v}\right)J^{-1}\\
 & = & J\left(\Phi_{\pi^{\prime}}^{r_{\pi\left(m-2\right)}r_{\pi\left(m\right)}v}\right)J^{-1}
\end{eqnarray*}

Since now $i=\pi\left(m\right)\in\pi^{\prime}\left(\left[m-1\right]\right)$,
we can deduce from the previous case that $\Phi_{\pi^{\prime}}^{r_{\pi\left(m-2\right)}r_{\pi\left(m\right)}v}|_{\st(\xi_{i})}=\id_{\st\left(\xi_{i}\right)}$
and since $\Phi_{\pi}^{v}$ and $\Phi_{\pi^{\prime}}^{r_{\pi\left(m-2\right)}r_{\pi\left(m\right)}v}$
are conjugates we get $\Phi_{\pi}^{v}|_{\st(\xi_{i})}=\id_{\st\left(\xi_{i}\right)}$.

Case 3. If $i=\pi\left(m\right)$ and $m=2$,

\begin{eqnarray*}
\Phi_{\pi}^{v} & = & \Phi_{\pi|_{\left[m-1\right]}}^{r_{\pi\left(m-1\right)}r_{\pi\left(m\right)}v}\circ\Phi_{\pi|_{\left[m-1\right]}}^{v}\\
 & = & \left(R_{r_{\pi\left(m\right)}v}^{-1}\circ R_{r_{\pi\left(m-1\right)}r_{\pi\left(m\right)}v}\right)\circ\left(R_{r_{\pi\left(m-1\right)}v}^{-1}\circ R_{v}\right)\\
 & = & J\left(\left(R_{r_{\pi\left(m-1\right)}r_{\pi\left(m\right)}v}\circ R_{r_{\pi\left(m-1\right)}v}^{-1}\right)\circ\left(R_{v}\circ R_{r_{\pi\left(m\right)}v}^{-1}\right)\right)J^{-1}
\end{eqnarray*}

When $J=R_{r_{\pi\left(m\right)}v}^{-1}$, the proof proceeds similarly
to the proof of Case 2 and the proof for $m=1$.
\end{proof}
\end{proof}

\section{Hyperplane automorphism extension property}

Recall from the introduction the following.
\begin{defn}
Let $X$ be a CAT(0) cube complex. A set of pairwise transverse hyperplanes
$\hhat_{1},\ldots,\hhat_{d}$ satisfy the \emph{hyperplane automorphism
extension property} (HAEP) if for all $\hat{f}\in\Aut\left(\hhat_{1}\cup\ldots\cup\hhat_{d}\right)$
there exists an automorphism $f\in\Aut X$ such that $f$ stabilizes
$\hhat_{1}\cup\ldots\cup\hhat_{d}$ and $f|_{\hhat_{1}\cup\ldots\cup\hhat_{d}}=\hat{f}$.\end{defn}
\begin{prop}
\label{prop: Hyperplane Automorphism Extension Property}Let $L$
be superstar-transitive, and let $X$ be the unique $L$-cube-complex.
Every transverse set of hyperplanes $\hhat_{1},\ldots,\hhat_{d}$
in $X$ satisfy the HAEP.\end{prop}
\begin{proof}
We first cut $X$ along $\hhat_{1},\ldots,\hhat_{d}$. Let $Z$ be
a component of the resulting complex. Let $x_{0}$ be a vertex of
$\hhat_{1}\cap\ldots\cap\hhat_{d}$ in $Z$. Let $\left\{ x_{n}\right\} _{n\le0}$
be some enumeration of the vertices of $\hhat_{1}\cup\ldots\cup\hhat_{d}$
in $Z$. Let $\left\{ x_{n}\right\} _{n>0}$ be an enumeration of
the remaining vertices of $Z$ with non-decreasing distance from $x_{0}$.
Similarly to Lemma \ref{lem: induction on Ccc} the enumeration $\left\{ x_{n}\right\} _{n}$
is admissible.

The map $\hat{f}|_{Z}$ defines a $0$-admissible map $(\hat{f}|_{Z},\{\left(\hat{f}|_{Z}\right)_{x_{i}}\}_{i\le0})$.
By Lemma \ref{lem:extension of admissible maps}, one can extend the
0-admissible map $(\hat{f}|_{Z},\{\left(\hat{f}|_{Z}\right)_{x_{i}}\}_{i\le0})$
to an automorphism $f_{Z}$ of $Z$. 

Finally, let $f$ be the automorphism of $X$ whose restriction to
each sector, $Z$, is $f_{Z}$.
\end{proof}
Recall from \cite{HaPa98} that for a group $G$ of automorphisms
of a CAT(0) cube complex $X$ we denote by $G^{+}$ the subgroup of
$G$ generated by all the elements that fix some halfspace of $X$,
i.e $G^{+}=\left\langle \Fix_{G}(\hhs)\,|\,\mbox{\ensuremath{\hhs}\ a halfspace of \ensuremath{X}}\right\rangle $.
\begin{lem}
\label{cor:extension of Aut+}Let $\hhat,\khat$ be a transverse pair
of hyperplanes which satisfy the HAEP, then each automorphism $f$
of $\hhat$ which fixes the halfspace $\khs\cap\hhat$ of $\hhat$ can
be extended to an automorphism $F$ of $X$ which fixes $\khs$.

In particular, if the hyperplane $\hhat$ satisfies the HAEP for any
transverse hyperplane $\khat$ then any element of $\Aut^{+}\hhat$
can be extended to an element of $\Aut^{+}X$.\end{lem}
\begin{proof}
Let $\hat{f}\in\Aut(\hhat\cup\khat)$ be the automorphism defined
by $\hat{f}|_{\hhat}=f,\,\hat{f}|_{\khat}=\id_{\khat}$. By the HAEP,
we can extend $\hat{f}$ to an automorphism, $f^{\prime}$, of $X$.
Finally, define $f$ to be the automorphism defined by $f|_{\khs^{*}}=f^{\prime}|_{\khs^{*}},\, f|_{\khs}=\id_{\khs}$.
\end{proof}

\section{Virtual simplicity of automorphism groups}
\begin{lem}
Let $G$ act transitively on a set $S$, and let $H$ be a subgroup
of $G$. If $S/H$ is finite and for some $x\in S$, $H_{x}=\Stab_{H}(x)$
has finite index in $G_{x}=\Stab_{G}(x)$, then $H$ has finite index
in $G$.\end{lem}
\begin{proof}
It follows from the following inequality $\left|G/H\right|\le\left|S/H\right|\cdot\left|G_{x}/H_{x}\right|<\infty$.\end{proof}
\begin{prop}
\label{prop: Aut+ has finite index}Let $X$ be a proper finite-dimensional
CAT(0) cube complex. Let $G=\Aut X$. Assume the following properties
hold:
\begin{enumerate}
\item There exists a hyperplane orbit $G\hhat$ such that $G^{+}$ has finitely
many orbits in $G\hhat$.
\item The group $\Aut^{+}\hhat$ has finite index in $\Aut\hhat$.
\item The hyperplane $\hhat$ satisfies the HAEP.
\item For every hyperplane $\khat$ transverse to $\hhat$ the pair $\hhat\cup\khat$
satisfies the HAEP.
\end{enumerate}
Then, $G^{+}$ has finite index in $G$.\end{prop}
\begin{proof}
Let $S=G\hhat$ and let $H=G^{+}$. The proposition will follow from
the previous lemma once we show that $G_{\hhat}^{+}=\Stab_{G^{+}}(\hhat)$
has finite index in $G_{\hhat}=\Stab_{G}(\hhat)$. 

By condition 3, we have that the restriction $G_{\hhat}|_{\hhat}$
is exactly $\Aut\hhat$. Similarly, by condition 4 and Lemma \ref{cor:extension of Aut+}
we deduce that $G_{\hhat}^{+}|_{\hhat}$ is excatly $\Aut^{+}\hhat$.
Thus, by condition 2, we get $[G_{\hhat}:G_{\hhat}^{+}]=[\Aut\hhat:\Aut^{+}\hhat]<\infty$. 
\end{proof}

\subsection{Examples of virtually simple automorphism groups}

\subsubsection{\label{sub:The-Kneser-complex}The Kneser complex and the associated
Davis complex}

We define the Kneser complex, $K(n,S)$, to be the simplicial complex
whose vertices are all subsets of size $n$ of $S$, and a collection
of vertices span a simplex if they represent pairwise disjoint subsets.

Let $n,d\in\mathbb{N},$ and let $S=\left\{ 1,\ldots,nd+1\right\} $.
The complex $L:=K_{n}^{d}=K(n,S)$ is a $(d-1)$-dimensional flag
simplicial complex. Since $\left|S\right|\ne2n$, the Erd\H{o}s-Ko-Rado
Theorem tells us that any automorphism of $L$ is induced from a permutation
of $S$ (see Corollary 7.8.2 in \cite{GoRo01}).
\begin{prop}
\label{prop:Kneser is superstar-transitive}The complex $L$ is superstar-transitive.\end{prop}
\begin{proof}
Let $k$ be a non-negative integer, let $\sigma,\sigma^{\prime}$
be two $k$-simplices of $L$ and let $\phi:\st(\sigma)\to\st(\sigma^{\prime})$
be an isomorphism. Since $L$ is clearly $\Delta^{k}$-transitive
(i.e. any isomorphism from one simplex to another can be extended
to an automorphism of $L$) we may assume without loss of generality
that $\sigma=\sigma^{\prime}=\left[v_{0},\ldots,v_{k}\right]$ and
$\phi$ fixes $\sigma$.

The map $\phi$ induces automorphisms on the links $\phi_{i}:\Lk(v_{i},L)\to\Lk(v_{i},L)$
for $i=0,\ldots,k$. For all $1\le i\le k$, the link $\Lk(v_{i},L)$
is naturally identified with $K(n,S\setminus v_{i})$ and thus $\phi_{i}$
is given by a permutation $\pi_{i}\in{\rm Sym}\left(S\setminus v_{i}\right)$.

If $k=0$, let $\pi$ be the permutation which fixes the set $v_{i}$
and restricts to $\pi_{i}$ on $S\setminus v_{i}$. 

If $k>0$, the maps $\phi_{i}$ and $\phi_{j}$ coincide along $\Lk\left(e_{i,j},L\right)$
where $e_{i,j}$ is the edge connecting $v_{i}$ and $v_{j}$. As
before $\Lk\left(e_{i,j},L\right)$ is naturally identified with $K\left(n,S\setminus\left(v_{i}\cup v_{j}\right)\right)$
and thus $\pi_{i}|_{S\setminus\left(v_{i}\cup v_{j}\right)}=\pi_{j}|_{S\setminus\left(v_{i}\cup v_{j}\right)}$.
Hence, the maps $\pi_{i}$ define a unique permutation $\pi\in{\rm Sym}(S)$
whose restriction to $S\setminus v_{i}$ is $\pi_{i}$ for all $0\le i\le k$.

In both cases the permutation $\pi$ defines an automorphism $\Phi$
of $L$ whose restriction to $\st(\sigma)$ is $\phi$.\end{proof}
\begin{lem}
\label{lem: Aut+(Kneser) is Aut(Kneser)}If $n\ge2$ then the group
$\Aut L$ is generated by all elements which fix the star of a vertex
in $L$. In particular, $\Aut^{+}D(L)$ acts transitively on hyperplanes.\end{lem}
\begin{proof}
Since $\Aut L={\rm Sym}(S)$ and since ${\rm Sym}(S)$ is generated
by transpositions, it suffices to show that all transpositions fix
a star of a vertex. But this is true since the transposition exchanging
$a,b\in S$ fixes the star of a vertex which contains $a$ and $b$.

To prove that there is only one orbit of hyperplanes, note that for
every $x\in D(L)$ and for every $f\in\Aut L$ there is an automorphism
$F_{x,f}$ of $D(L)$ that fixes $x$ and at each vertex the induced
automorphism on the links is $f$. Moreover, if $f$ fixes a star
of a vertex then $F_{x,f}$ fixes the corresponding hyperplane near
$x$. By the above, for all $x\in X$ and $f\in\Aut L$, $F_{x,f}$
is in $\Aut^{+}D(L)$. Since $\Aut L$ acts transitively on the vertices
of $L$ we get that $F_{x,f}$ acts transitively on the hyperplanes
near $x$. This holds for all $x\in D(L)$, and any two adjacent vertices
are contained in the carrier of a hyperplane, and therefore $\Aut^{+}D(L)$
acts transitively on hyperplanes.\end{proof}
\begin{cor}
\label{cor: virtual simplicity for Kneser}Let $n\ge2,d\ge1$ and
let $L=K(n,\left\{ 1,\ldots,nd+1\right\} )$. The group $G=\Aut D(L)$
is virtually simple.\end{cor}
\begin{proof}
We begin by showing that $G^{+}$ is simple by verifying the assumptions
of Corollary \ref{Simplicity of Aut(X)} and Claim \ref{criteria for faithful action}. 

For all finite flag simplicial $L$ the right-angled Davis complexes
$D(L)$ is proper, finite dimensional and cocompact (since the Coxeter
group $W_{L}$ acts cocompactly on $D(L)$). The complex $D(L)$ is
essential since $L$ is not the star of any of its vertices. Similarly
every sector $\hhs\cap\khs$ in $D(L)$ contains a hyperplane because
$L$ is not the star of any of its edges. The complex $D(L)$ is irreducible
since $L$ is not a join of two subcomplexes, because the complement
graph (the graph of non-empty intersections of subsets of size $n\ge2$
in $\left\{ 1,\ldots,nd+1\right\} $) is connected. The group $G$
is non-elementary because $\partial D(L)=\Lambda G$ contains more
than 2 points (because $L$ contains an independant set of vertices
of size 3) and $G$ acts without a fixed point at infinity (the Coxeter
group acts with inversions along hyperplanes, thus does not fix a
point in $\partial D(L)$).

In order to prove the corollary it suffices to show that $G^{+}$
is of finite index in $G$. We prove it by induction on $d\ge0$.
The base case $d=0$, is trivial since the complex $D(L)$ is a single
vertex, thus $G$ is trivial.

The conditions of Proposition \ref{prop: Aut+ has finite index} hold,
Condition 1 by Lemma \ref{lem: Aut+(Kneser) is Aut(Kneser)}, Condition
2 by the induction hypothesis (the hyperplane are isomorphic to the
Davis complex associated to $K\left(n,\left\{ 1,\ldots,n\left(d-1\right)+1\right\} \right)$),
and Conditions 3 and 4 by Propositions \ref{prop: Hyperplane Automorphism Extension Property}
and \ref{prop:Kneser is superstar-transitive}.
\end{proof}

\subsubsection{Superstar-transitive graphs and unique square complexes}
\begin{lem}
Let $L$ be a finite, connected, flag, simplicial, superstar-transitive
graph. If all the vertices in $L$ have degree $\ge3$, then the subgroup,
$\Aut^{+}(L)$, of $\Aut(L)$ generated by the automorphisms which
fix the star of a vertex has at most two vertex orbits. Moreover,
these orbits form a partition of the graph. 

In particular $\Aut^{+}D(L)$ has at most two orbits of hyperplanes.\end{lem}
\begin{proof}
If $v_{1},v_{2},w$ are adjacent to $v$ then, by $\st(\Delta^{1})$-transitivity,
there exists an automorphism exchanging $v_{1},v_{2}$ and fixing
the star of $w$. Thus all the adjacent vertices of $v$ are in the
same orbit of $\Aut^{+}(L)$. This holds for all $v$, and by connectivity
we get the desired conclusion.

As in the proof of Lemma \ref{lem: Aut+(Kneser) is Aut(Kneser)},
for all $x\in D(L)$, $\Aut^{+}D(L)$ has at most two orbits of hyperplanes
whose carrier contains $x$. Using the fact that any two adjacent
vertices are in the carrier of two transverse hyperplanes we deduce
that $\Aut^{+}D(L)$ has at most two orbits of hyperplanes in $D(L)$.
\end{proof}
This Lemma allows us to deduce, as in the proof of Corollary \ref{cor: virtual simplicity for Kneser}
the following.
\begin{thm}
\label{thm:v.simplicity of square complexes}Let $L$ be a finite,
connected, flag, simplicial, superstar-transitive graph all of whose
vertices have degree $\ge3$ and which is not a complete bi-partite
graph. Let $X=D(L)$ be the unique CAT(0) square complex whose vertex
links are isomorphic to $L$. Then $\Aut(X)$ is virtually simple.
\end{thm}
\appendix

\counterwithin{thm}{section}
\counterwithin{figure}{section}

\section{Appendix: Simplicity of automorphism groups of rank one cube complexes}

Let $T$ be a tree, and let $Aut(T)$ be the automorphism group of
$T$. In \cite{Tit70}, J. Tits showed that under certain conditions
the subgroup $Aut^{+}(T)$ of $Aut(T)$ generated by the fixators
of edges of $T$ is simple. In \cite{HaPa98}, F. Paulin
and F. Haglund generalized Tits' result for Gromov hyperbolic cube
complexes. In this appendix, we further generalize these results to
rank one CAT(0) cube complexes. A similar result for (not necessarily
locally finite) thick right-angled buildings was established by P-E.
Caprace in \cite{Cap14}. 

Let $X$ be a CAT(0) metric space. A rank one isometry is a hyperbolic
isometry $g\in\Isom(X)$ none of whose axes bounds a flat half-plane
(i.e. a subspace which is isometric to the Euclidean half plane).
Any hyperbolic element in a Gromov hyperbolic cube complex is such
since there are no flat half-spaces in a hyperbolic cube complex.
In general, rank one elements in locally compact CAT(0) cube complexes
act on the boundary $\partial X$ with a north-south dynamics similarly
to the action of hyperbolic elements in hyperbolic spaces. For $G\subset\Isom(X)$
we denote by $\Lambda(G)$ the \textit{limit set} of $G$ in $\partial X$,
i.e., the set of accumulation points in $\partial X$ of an orbit
of $G$. The group $G$ is called elementary if either $|\Lambda(G)|\le2$
or $G$ fixes a point at infinity. Let $X$ be a proper CAT(0) space,
and let $G\le\Isom(X)$ be a non-elementary subgroup which contains
a rank one element, then the set of pairs of fixed points in $\partial X$
of rank one elements is dense in the complement of the diagonal $\Delta$
of $\Lambda(G)\times\Lambda(G)$ (see \cite{Ham09}).

Let $X$ be a CAT(0) cube complex. Let $\Hs$ be the set of all half-spaces
of $X$, and let $\Hhat$ be the set of corresponding hyperplanes.
We denote by $\hat{}:\Hs\to\Hhat$ the natural map mapping each half-space
to its bounding hyperplane, and by $^{*}:\Hs\to\Hs$ the map sending
a half-space to its complementary half-space. Recall the following
definitions from \cite{CaSa11}. We say that $X$ is \emph{irreducible}
if it cannot be expressed as a (non-trivial) product. We say that
$X$ is \emph{essential} if every half-space $\hhs\in\Hs$ contains
points arbitrarily far from $\hhat$. Let $G\le\Aut(X)$ a group of
automorphisms of $X$. We say that $X$ is $G$-\emph{essential} if
every half-space of $X$ contains $G$-orbit points arbitrarily far
from its bounding hyperplane. In \cite{CaSa11}, Caprace
and Sageev proved the following rank rigidity result. 
\begin{thm*}[Rank Rigidity for CAT(0) cube complexes, \cite{CaSa11}]
If $X$ is a finite-dimensional irreducible CAT(0) cube complex,
and $G\le\Aut(X)$ acts essentially on $X$ without fixed points at
infinity, then $G$ contains a rank one isometry.
\end{thm*}

\begin{defn}
Let $X$ and $G$ be as above. We denote by $G^{+}$ the subgroup
of $G$ generated by the fixators of half-spaces of $X$, i.e.,
\[
G^{+}=\left\langle g\in Fix_{G}(\mathfrak{h})\,\middle|\,\mathfrak{h}\in\mathcal{H}(X)\right\rangle 
\]

We recall from \cite{Tit70,HaPa98} that the action of
$G$ on $X$ satisfies \emph{property (P)} if for every nested sequence
of half-spaces $(\mathfrak{h}_{n})_{n\in\mathbb{Z}}\subset\mathcal{H}$,
$\mathfrak{h}_{n+1}\subset\mathfrak{h}_{n}$, the following map is
an isomorphism:

\begin{equation}
\Fix_{G}\left(\bigcup_{n}\hat{\mathfrak{h}}_{n}\right)\to\prod_{n}\Fix_{G}\left(\hat{\mathfrak{h}}_{n}\cup\hat{\mathfrak{h}}_{n+1}\right)|_{\mathfrak{h}_{n}\cap\mathfrak{h}_{n+1}^{*}}\label{eq: property P}
\end{equation}

Where $\Fix_{G}\left(\hat{\mathfrak{h}}_{n}\cup\hat{\mathfrak{h}}_{n+1}\right)|_{\mathfrak{h}_{n}\cup\mathfrak{h}_{n+1}^{*}}$
is the image of $\Fix_{G}\left(\hat{\mathfrak{h}}_{n}\cup\hat{\mathfrak{h}}_{n+1}\right)$
in $\Aut\left(\mathfrak{h}_{n}\cap\mathfrak{h}_{n+1}^{*}\right)$
under the restriction map.\end{defn}
\begin{rem*}
Note that $G=\Aut(X)$ satisfies property (P), and in this case $G^{+}$
is generated by the fixators of the carriers of hyperplanes.
\end{rem*}
We prove the following.
\begin{thm}
\label{thm:Simplicity}Let $X$ be a proper finite-dimensional irreducible
CAT(0) cube complex, and let $G\le\Aut(X)$ be a non-elementary group
acting essentially on $X$ with property (P) and $\Lambda(G)=\partial X$.
Then for all $N\vartriangleleft G^{+}$ either $N=G^{+}$ or $N$
acts trivially on $\partial X$. In particular, if we further assume
that $G$ acts faithfully on $\partial X$, then $G^{+}$ is either
simple or trivial.
\end{thm}
By the remark above, we deduce:
\begin{cor}
\label{Simplicity of Aut(X)}Let $X$ be a proper finite-dimensional
irreducible essential CAT(0) cube complex with co-compact, non-elementary
$\Aut(X)$ action. Then for all $N\vartriangleleft\Aut^{+}(X)$ either
$N=\Aut^{+}(X)$ or $N$ acts trivially on $\partial X$. In particular,
if we further assume that $\Aut(X)$ acts faithfully on $\partial X$,
then $\Aut^{+}(X)$ is either simple or trivial.
\end{cor}

\begin{defn}
For $\mathfrak{h}\in\mathcal{H}(X)$ let 
\[
\mathfrak{h}_{\infty}=\left\{ \xi\in\partial X\,|\, r\cap\mathfrak{h}\ne\emptyset,\forall\mbox{geodesic rays }r\mbox{ such that }r(\infty)=\xi\right\} 
\]

It can also be defined as the collection of points in $\partial X$
that are not accumulation of $\hhs^{*}$. And for $\hat{\mathfrak{h}}\in\mathcal{H}(X)$
let $\hat{\mathfrak{h}}_{\infty}$ be the set of accumulation points
of $\hat{\mathfrak{h}}$ in $\partial X$.

Recall from \cite{BeCh12} that two hyperplanes $\hhat,\khat$
are \emph{strongly separated }if there is no hyperplane which intersects
both of them.\end{defn}
\begin{prop}[Proposition 5.1 in \cite{CaSa11}]
\label{strongly separated hyperplane}Under the same assumptions
as in Theorem \ref{thm:Simplicity}, for all $\hhs\in\Hs$ there exists
a half-space $\khs\subset\hhs$ such that $\hhat$ and $\khat$ are strongly
separated. In particular, the open set $\hhs_{\infty}$ is non-empty.
\end{prop}
We remark that under mild assumptions the action of $\Aut(X)$ on
$\partial X$ is faithful, as the following claim shows.
\begin{claim}
\label{criteria for faithful action}Under the same assumptions as
in Theorem \ref{thm:Simplicity}, assume moreover that the intersection
of any pair of crossing halfspaces $\hhs_{1},\hhs_{2}$ contains a halspace
$\khs\subset\hhs_{1}\cap\hhs_{2}$. Then $\Aut(X)$ acts faithfully on $\partial X$.\end{claim}
\begin{proof}
Let $1\ne g\in Aut(X)$, and let $\mathfrak{h}\in\mathcal{H}(X)$
be a half-space such that $g\mathfrak{h}\ne\mathfrak{h}$. Then $g$
and $\mathfrak{h}$ satisfy one of the following cases:
\begin{casenv}
\item $g\mathfrak{h}\subset\mathfrak{h}^{*}$. Then $g$ sends the corresponding
$\mathfrak{h}_{\infty}$ into $\mathfrak{h}_{\infty}^{*}$ which are
non-empty and disjoint.
\item $g\mathfrak{h}\subset\mathfrak{h}$. Let $\hat{\mathfrak{l}}\in\hat{\mathcal{H}}(X)$
be a hyperplane transversal to $\hat{\mathfrak{h}}$, (we may assume
without loss of generality that $\forall n\in\mathbb{N}$ $g^{n}\mathfrak{h}\nsubseteq\mathfrak{l}$,
by determining the orientation of $\mathfrak{l}$ so that $g^{m}\mathfrak{h}\nsubseteq\mathfrak{l}$
for the minimal $m$ such that $g^{n}\hat{\mathfrak{h}}\cap\hat{\mathfrak{l}}=\emptyset$).
By our assumption, let $\mathfrak{k}\subset\mathfrak{l}\cap\mathfrak{h}$.
Then either $\forall n\in\mathbb{N},\, g^{n}\hhat\subset\mathfrak{k}^{*}$
or $\exists n\in\mathbb{N},\, g^{n}\hat{\mathfrak{h}}\cap\hat{\mathfrak{k}}\ne\emptyset$.
If $\forall n,\, g^{n}\mathfrak{\hhat}\subset\mathfrak{k}^{*}$ then
there exists $r$ such that $\mathfrak{k}\subset g^{r-1}\mathfrak{h}\cap g^{r}\mathfrak{h}^{*}$,
hence $g$ takes $\mathfrak{k}_{\infty}$ into $g\mathfrak{k}_{\infty}$
(which are disjoint since $\mathfrak{k}_{\infty}\subset g^{r}\mathfrak{h}_{\infty}^{*}$
and $g\mathfrak{k}_{\infty}\subset g^{r}\mathfrak{h}_{\infty}$).
If $\exists n,\, g^{n}\hat{\mathfrak{h}}\cap\hat{\mathfrak{k}}\ne\emptyset$
then we may assume that $n$ is minimal. By our assumption there exists
$\mathfrak{k}_{1}\subset\mathfrak{k}\cap g^{n}\mathfrak{h}^{*}$.
Therefore $\mathfrak{k}_{1}\subset g^{n-1}\mathfrak{h}\cap g^{n}\mathfrak{h}^{*}$,
and $g$ acts non-trivially on $\partial X$ as before.
\item $g\hat{\mathfrak{h}}\cap\hat{\mathfrak{h}}\ne\emptyset$. By our assumption
let $\mathfrak{k}\subset\mathfrak{h}^{*}\cap g\mathfrak{h}$, then
$g$ sends $\mathfrak{k}_{\infty}$ into $g\mathfrak{h}_{\infty}^{*}$.
\end{casenv}
\end{proof}
For completeness we prove the following:
\begin{lem}[Lemme 6.4 in \cite{HaPa98}]
\label{lem:F=00003D[g,F]}Let $g\in Aut(X)$, and $\mathfrak{h}\in\mathcal{H}(X)$
such that $g\mathfrak{h}\subset\mathfrak{h}$, and let $F={\rm Fix}_{G}(\bigcup_{n\in\mathbb{Z}}g^{n}\hat{\mathfrak{h}})$,
then $F=[g,F]$.\end{lem}
\begin{proof}
Clearly $F\supset[g,F]$. Now, let $f\in F.$ we will show that there
exist $f^{\prime}\in F$ such that $f=f^{\prime-1}g^{-1}f^{\prime}g$.

By the property (P) assumption we can define $f^{\prime}$ by its
restrictions to the sets $g^{n}\mathfrak{h}\cap g^{n+1}\mathfrak{h}^{*}$.
We do so by induction on $n$:

For $n=0$, define 
\[
f^{\prime}|_{\mathfrak{h}\cap g\mathfrak{h}^{*}}=id_{\mathfrak{h}\cap g\mathfrak{h}^{*}}
\]

For $n>0$, define 
\[
f^{\prime}|_{g^{n}\mathfrak{h}\cap g^{n+1}\mathfrak{h}^{*}}=gf^{\prime}fg^{-1}
\]

Similarly define $f^{\prime}$ for $n<0$.

Notice that this is well-defined: if $x\in g^{n}\mathfrak{h}\cap g^{n+1}\mathfrak{h}^{*}$,
then $fg^{-1}(x)\in g^{n-1}\mathfrak{h}\cap g^{n}\mathfrak{h}^{*}$,
and therefore $f^{\prime}(f(g(x))$ is defined by the induction hypothesis.
Now observe that $f^{\prime}$ has the desired property.\end{proof}
\begin{prop}
\label{pro:boundary of a normal subgroup}Let $X$ be a proper CAT(0)
space, and let $H\le\Isom(X)$ be non-elementary. Assume $H$ contains
a rank one isometry and acts non-elementary on $X$ with $\Lambda(H)=\partial X$,
and let $N\vartriangleleft H$, then either $N$ acts trivially on
$\partial X$, or $\Lambda(N)=\Lambda(H)=\partial X$ and $N$ is
non-elementary.\end{prop}
\begin{proof}
Assume that $N$ does not act trivially on $\partial X$. First we
show that the limit set $\Lambda(N)$ is either $\Lambda(H)$ or empty.

Let $\xi\in\Lambda(N)$ and $h\in H$. There exists a sequence $n_{k}\in N$
such that for all $x\in X$, $n_{k}x\to\xi$. Apply $h\in H$ to $\xi$.
By normality of $N$, we get a sequence $n_{k}^{\prime}=hn_{k}h^{-1}\in N$.
\[
h\xi\leftarrow hn_{k}.x=n_{k}^{\prime}.\left(hx\right)
\]

Therefore $h\xi\in\Lambda(N)$; hence $\Lambda(N)$ is $H$-invariant.
By minimality of the action of $H$ on $\Lambda(H)$ (see \cite{Ham09})
we get $\Lambda(N)=\Lambda(H)\mbox{ or }\emptyset$. 

To show that $\Lambda(N)=\Lambda(H)$ assume for contradiction that
$\Lambda(N)=\emptyset$. Then the action of $N$ on $X$ is bounded;
hence has a fixed point $x_{0}$. Since $N$ acts non-trivially on
$\partial X$, there exists $\xi\in\partial X$ and $n\in N$ such
that $n.\xi\ne\xi$. By $\partial X=\Lambda(H)$, there exists a sequence
$h_{k}\in H$ such that $h_{k}.x_{0}\to\xi$. By applying $n$ we
get $nh_{k}.x_{0}\to n.\xi$. On the other hand, by normality, we
have $nh_{k}=h_{k}n_{k}^{\prime}$ (for some $n_{k}^{\prime}\in N$).
Thus, $nh_{k}.x_{0}=h_{k}n_{k}^{\prime}.x_{0}=h_{k}.x_{0}\to\xi$.
Therefore, we get $\xi=n.\xi$ which contradicts our assumption.

To see that $N$ is non-elementary, we are left to show that $N$
does not fix a point at infinity. Assume by contradiction that $\xi\in\partial X$
is $N$-fixed, then by normality $g\xi$ is $N$-fixed for all $g\in H$.
By minimality we get that $\Lambda(H)=\partial X$ is $N$-fixed.
Hence $N$ acts trivially on $\partial X$. But we assumed that $N$
acts non-trivially on $\partial X$.\end{proof}
\begin{rem*}
Without assuming $\partial X=\Lambda\left(H\right)$, the same argument
shows that for every normal subgroup $N\vartriangleleft H$ either
$N$ acts trivially on $\Lambda(H)$ or $\Lambda\left(N\right)=\Lambda\left(H\right)$.
\end{rem*}
We shall now prove Theorem \ref{thm:Simplicity}:
\begin{proof}

\begin{figure}
\centering
\includegraphics[width=0.5\textwidth]{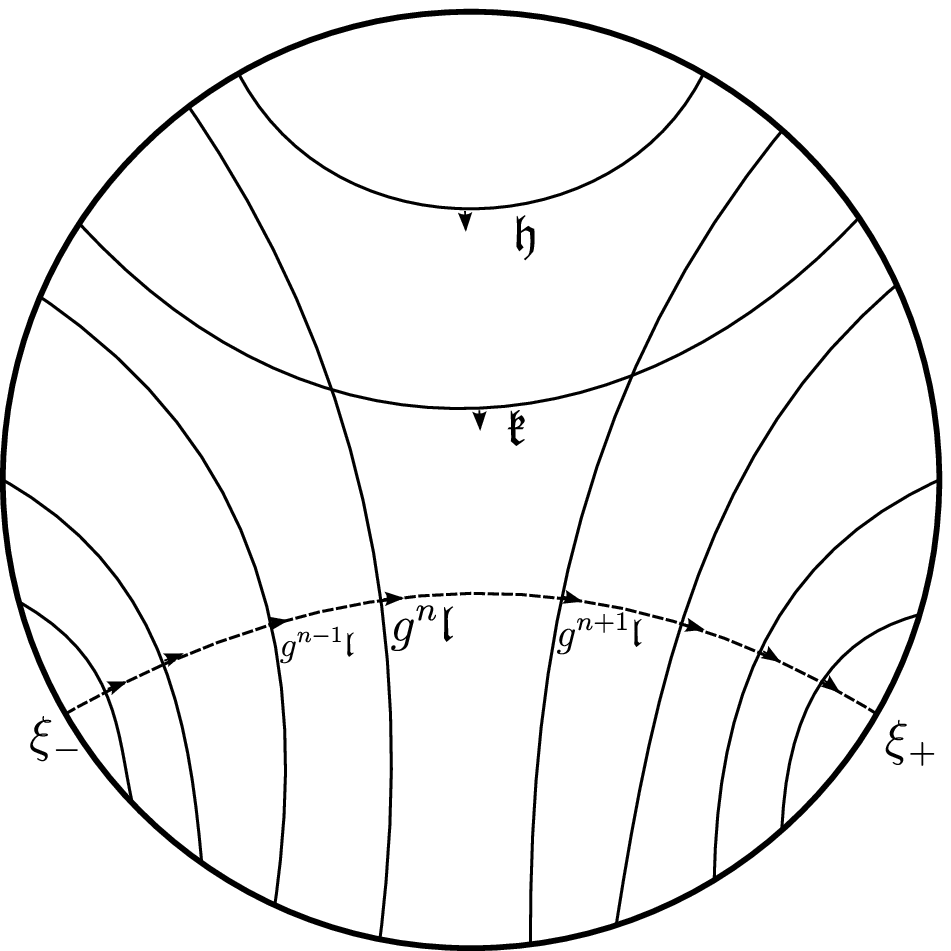}

\caption{ The action of $g$ on $X$ and the halfspaces $\hhs,\,\khs$ and $\mathfrak{l}$}
\label{fig:simplicity}
\end{figure}

Let $N\vartriangleleft G^{+}$ and assume $N$ acts non-trivially
on $\partial X$. In order to prove the theorem, it suffices to show
that $\Fix_{G}(\mathfrak{h})\subset N$ for all $\mathfrak{h}\in\mathcal{H}(X)$.
Let $\mathfrak{h}\in\mathcal{H}(X)$. Apply Proposition \proref{boundary of a normal subgroup}
and the rank rigidity theorem first on $G^{+}\vartriangleleft G$
and then on $N\vartriangleleft G^{+}$, to obtain that $\Lambda(N)=\partial X$,
$N$ is non-elementary and contains a rank one isometry. By Claim
\ref{strongly separated hyperplane} there exists a half-space $\khs\subset\hhs$
such that $\hhat$ and $\khat$ are strongly separated. The set $\partial^{2}X\cap\left(\khs_{\infty}\times\khs_{\infty}\right)$
is a non-empty open set in $\partial^{2}X$, hence, by Theorem 1.1(2)
of \cite{Ham09}, there exists a rank one isometry $g$ whose
two fixed points $(\xi_{+},\xi_{-})$ in $\partial X$ are in $\partial^{2}X\cap\left(\khs_{\infty}\times\khs_{\infty}\right)$.
By passing to a power of $g$ we may further assume that there exists
$\mathfrak{l}\in\mathcal{H}(X)$ such that $g\mathfrak{l}\subset\mathfrak{l}$.
See Figure \ref{fig:simplicity}.

Let $F={\rm Fix}_{G}(\bigcup_{n\in\mathbb{Z}}g^{n}\hat{\mathfrak{l}})$.
By the above we see that $\bigcup_{n\in\mathbb{Z}}g^{n}\hat{\mathfrak{l}}\subset\hhs$.
Therefore, by Lemma \ref{lem:F=00003D[g,F]} we have:
\[
\Fix_{G}(\mathfrak{h})\subset F=[g,F]\subset N
\]
\end{proof}
\begin{rem*}
In fact, one can assume a weaker version of property (P). For example,
assuming that for every element $g\in G$ and $\mathfrak{l}\in\Hs(X)$
such that $g\mathfrak{l}\subset\mathfrak{l}$, the map \ref{eq: property P}
is an isomorphism for the collection $\left\{ g^{n}\mathfrak{\hat{l}}\right\} _{n\in\mathbb{Z}}$.\end{rem*}

\bibliographystyle{plain}
\bibliography{Bibliography}

{\noindent
Nir Lazarovich. Email: nirl@tx.technion.ac.il\\Department of Mathematics, Technion, Haifa 32000, Israel}

\end{document}